\documentclass[11pt]{amsart}
\usepackage{amssymb,amsmath}

\makeatletter
\theoremstyle{plain}
 \newtheorem{thm}{Theorem}[section]
\newtheorem{thm*}{Theorem}
 
 \newtheorem{prop}[thm]{Proposition}
 
 \numberwithin{equation}{section} 
\numberwithin{figure}{section} 
 \theoremstyle{plain}
 \theoremstyle{definition}
 \newtheorem{defn}[thm]{Definition}

\newcommand{\C}{{{\mathbb C}}}
\newcommand{\R}{{{\mathbb R}}}

\newcommand{\bH}{{{\bf H}}}
\newcommand{\bp}{{{\bf p}}}

\newcommand{\fp}{{{\mathfrak p}}}

\newcommand{\h}{{{\mathbb H}}}
\newcommand{\fX}{{{\mathfrak X}}}
\newcommand{\X}{{{\mathbb X}}}

\newcommand{\fH}{{{\mathfrak H}}}

\newcommand{\binfty}{{{\bf \infty}}}

\makeatother

\begin{document}

\title[Ptoleaean Inequality]{Complex cross--ratios and the ptolemaean inequality}

\author[I.D. Platis]{Ioannis D. Platis}

\begin{abstract}
We use Kor\'anyi--Reimann complex cross--ratios to prove the Ptolemaean inequality
and the Theorem of Ptolemaeus in the setting of the boundary of complex hyperbolic space and
the first Heisenberg group.
  
\end{abstract}

\address{Department of Mathematics, Univerity of Crete, Greece}
\email{jplatis@math.uoc.gr}
\date{12 May 2012}

\maketitle

\section{Introduction}

The Theorem of Ptolemaeus in planar Euclidean geometry states that the product of the euclidean lengths of the diagonals of an inscribed quadrilateral equals to the sum of the products of the euclidean lengths of its opposite sides. When one vertex of the quadrilateral does not lie on the circle passing from the other three verices, then we have inequality, known as the Ptolemaean inequality.

The Ptolemaean inequality and its generalisation to various spaces has recently been the study of many authors, for example see the innovative paper of S. Buyalo and V. Schroeder, \cite{BS} and the work of S.M. Buckley, K. Falk and D.J. Wraith in CAT(0) spaces, \cite{BFW}. In the present paper we use the Kor\'anyi--Reimann complex cross--ratios to give a proof of the Ptolemaean Inequality and the Theorem of Ptolemaeus in the boundary of complex hyperbolic space $\partial\bH^2_\C$ (Theorem \ref{thm:ptolemaean-ineq}) and accordingly in the Heisenberg group (Theorem \ref{thm:ptolemaean-heis}).

\section{Preliminaries}

\subsection{Complex Hyperbolic Space}\label{sec:chs}

Let $\mathbb{C}^{2,1}$ be the vector space $\mathbb{C}^{3}$ with
the Hermitian form of signature $(2,1)$ given by
$$
\left\langle {\bf {z}},{\bf {w}}\right\rangle 
={\bf w}^*J{\bf z}
=z_{1}\overline{w}_{3}+z_{2}\overline{w}_{2}+z_{2}\overline{w}_{1}
$$
with matrix 
$$
J=\left[\begin{array}{ccc}
0 & 0 & 1\\
0 & 1 & 0\\
1 & 0 & 0\end{array}\right].
$$
We consider the following subspaces of ${\mathbb C}^{2,1}$:
\begin{eqnarray*}
V_- & = & \Bigl\{{\bf z}\in{\mathbb C}^{2,1}\ :\ 
\langle{\bf z},\,{\bf z} \rangle<0\Bigr\}, \\
V_0 & = & \Bigl\{{\bf z}\in{\mathbb C}^{2,1}-\{{\bf 0}\}\ :\ 
\langle{\bf z},\,{\bf z} \rangle=0\Bigr\}.
\end{eqnarray*}

Let ${\mathbb P}:{\mathbb C}^{2,1}-\{0\}\longrightarrow {\mathbb C}P^2$ be
the canonical projection onto complex projective space. Then 
{\sl complex hyperbolic space} ${\bf H}_{\mathbb{C}}^{2}$
is defined to be ${\mathbb P}V_-$ and its boundary
$\partial{\bf H}^2_{\mathbb C}$ is ${\mathbb P}V_0$.
Specifically, ${\mathbb C}^{2,1}-\{{\bf 0}\}$ may be covered with three charts
$H_1$, $H_2$, $H_3$ where $H_j$ comprises those points in 
${\mathbb C}^{2,1}-\{{\bf 0}\}$ for which $z_j\neq 0$. It is clear that
$V_-$ is contained in $H_3$. The canonical projection
from $H_3$ to ${\mathbb C}^2$ is given by 
${\mathbb P}({\bf z})=(z_1/z_3,\,z_2/z_3)=z$. Therefore we can write
${\bf H}^2_{\mathbb C}={\mathbb P}(V_-)$ as
$$
{\bf H}^2_{\mathbb C} = \left\{ (z_1,\,z_2)\in{\mathbb C}^2
\ : \ 2\Re(z_1)+|z_2|^2<0\right\}.
$$
There are distinguished points in $V_0$ which we denote by 
${\bf o}$ and $\binfty$:
$$
{\bf o}=\left[\begin{matrix} 0 \\ 0 \\ 1 \end{matrix}\right], \quad
\binfty=\left[\begin{matrix} 1 \\ 0 \\ 0 \end{matrix}\right].
$$
Then $V_0-\{{\binfty}\}$ is contained in $H_3$ and 
$V_0-\{{\bf o}\}$ (in particular $\infty$) is contained
in $H_1$. Let ${\mathbb P}{\bf o}=o$ and ${\mathbb P}\binfty=\infty$. 
Then we can write $\partial{\bf H}^2_{\mathbb C}={\mathbb P}(V_0)$ as
$$
\partial{\bf H}^2_{\mathbb C}-\{\infty\} 
=\left\{ (z_1,\,z_2)\in{\mathbb C}^2
\ : \ 2\Re(z_1)+|z_2|^2=1\right\}.
$$
In particular $o=(0,0)\in\C^2$. 
In this manner, ${\bf H}^2_{\mathbb C}$ is the Siegel domain in 
${\mathbb C}^2$; see \cite{Gol}.

Conversely, given a point $z$ of 
${\mathbb C}^2={\mathbb P}(H_3)\subset{\mathbb C}P^2$ we may 
lift $z=(z_1,z_2)$ to a point ${\bf z}$ in $H_3\subset{\mathbb C}^{2,1}$, 
called the {\sl standard lift} of $z$, by writing ${\bf z}$ in non-homogeneous
coordinates as
$$
{\bf z}=\left[\begin{matrix} z_1 \\ z_2 \\ 1 \end{matrix}\right].
$$

The {\sl Bergman metric} on  ${\bf H}_{\mathbb{C}}^{2}$ is defined by the
distance function $\rho$ given by the formula
$$
\cosh^{2}\left(\frac{\rho(z,w)}{2}\right)
=\frac{\left\langle {\bf {z}},{\bf {w}}\right\rangle 
\left\langle {\bf {w}},{\bf {z}}\right\rangle }
{\left\langle {\bf {z}},{\bf {z}}\right\rangle 
\left\langle {\bf {w}},{\bf {w}}\right\rangle }
=\frac{\bigl|\langle {\bf z},{\bf w}\rangle\bigr|^2}
{|{\bf z}|^2|{\bf w}|^2}
$$
where ${\bf z}$ and ${\bf w}$ in $V_-$ are the standard lifts of $z$ and $w$ 
in ${\bf H}^2_{\mathbb C}$ and
$|{\bf z}|=\sqrt{-\langle{\bf z},{\bf z}\rangle}$.
Alternatively,
$$
ds^{2}=-\frac{4}{\left\langle {\bf {z}},{\bf {z}}\right\rangle ^{2}}
\det\left[\begin{array}{cc}
\left\langle {\bf {z}},{\bf {z}}\right\rangle  
& \left\langle d{\bf {z}},{\bf {z}}\right\rangle \\
\left\langle {\bf {z}},d{\bf {z}}\right\rangle  
& \left\langle d{\bf {z}},d{\bf {z}}\right\rangle \end{array}\right].
$$
The holomorphic sectional curvature of 
$ {\bf H}_{\mathbb{C}}^{2}$
equals to $-1$ and its real sectional curvature
is pinched between $-1$ and $-1/4$.

\subsubsection{Isometries, complex lines, Lagrangian planes}

Let ${\rm U(2,1) }$ be the group
of unitary matrices for  the  Hermitian form 
$\left\langle \cdot,\cdot\right\rangle $. Each such matrix $A$ satisfies
the relation $A^{-1}=JA^{*}J$ where $A^{*}$ is the Hermitian transpose of $A$.

The full group of holomorphic isometries of 
complex hyperbolic space is the \textsl{projective
unitary group} ${\rm PU(2,1)}={\rm U(2,1)}/{\rm U(1)}$, where 
${\rm U(1)}=\{ e^{i\theta}I,\theta\in[0,2\pi)\}$
and $I$ is the $3\times3$ identity matrix. 
We find sometimes  convenient to consider instead the group ${\rm SU(2,1)}$
of matrices which are unitary with respect to 
$\left\langle \cdot,\cdot\right\rangle $,
and have determinant $1$. 
Therefore ${\rm PU(2,1)}={\rm SU(2,1)}/\{ I,\omega I,\omega^{2}I\}$,
where $\omega$ is a non real cube root of unity, and so ${\rm SU}(2,1)$
is a 3-fold covering of ${\rm PU}(2,1)$. 

A complex line is an isometric image of the embedding of $\bH^1_\C$ into $\bH^2_\C$. A Lagrangian plane is an isometric image of $\bH^2_\R$ into $\bH^2_\C$.

\subsection{The boundary--Heisenberg group}\label{sec:boundary}

A finite point $z$ is in the boundary of the Siegel domain if its standard
lift to $\C^{2,1}$ is ${\bf z}$ where
$$
{\bf z}=\left[\begin{matrix} z_1 \\ z_2 \\ 1 \end{matrix}\right]
\quad \text{ where }\quad z_1 + \bar{z}_{1} + |z_2|^2 = 0.
$$
We write $z=z_2/\sqrt{2}\in\C$ and this 
condition becomes $2\Re(z_1)=-2|z|^2$. Hence we may
write $z_1=-|z|^2+it$ for $t\in\R$. That is for $z\in\C$ and
$t\in\R$:
$$
{\bf z}
=\left[\begin{matrix} -|z|^2+it \\ \sqrt{2}z\\ 1\end{matrix}\right]
$$
Therefore we may identify the boundary of the Siegel domain with
the one point compactification of $\C\times\R$.

The action of the stabiliser of infinity ${\rm Stab}(\infty)$ gives to the set of these points the structure of a non Abelian group. This is the Heisenberg group $\fH$ which is $\C\times\R$ with group law
$$
(z,t)*(w,s)=(z+w,t+s+2\Im(\overline{w}z)).
$$

The Heisenberg norm (Kor\'anyi gauge) is given by
$$
\left|(z,t)\right|=\left| |z|^2-it\right|^{1/2}.
$$
From this norm we obtain  a metric, the Kor\'anyi--Cygan (K--C) metric, on $\fH$ by the relation
$$
d_K\left((z_1,t_1),\,(z_2,t_2)\right)
=\left|(z_1,t_1)^{-1}*(z_2,t_2)\right|.
$$
Or, in other words
$$
d_K\left((z_1,t_1),\,(z_2,t_2)\right)
=\left| |z_1-z_2|^2-it_1+it_2-2i\Im(z_1\bar{z}_2)\right|^{1/2}.
$$
By taking the standard lift of points on $\partial{\bf H}^2_\C-\{\infty\}$ 
to $\C^{2,1}$ we can write the K--C metric as:
$$
d_K\left((z_1,t_1),\,(z_2,t_2)\right)
=\left|\left\langle\left[ \begin{matrix}
-|z_1|^2+it_1 \\ \sqrt{2}z_1 \\ 1 \end{matrix}\right],\,
\left[ \begin{matrix}
-|z_2|^2+it_2 \\ \sqrt{2}z_2 \\ 1 \end{matrix}\right]
\right\rangle\right|^{1/2}.
$$
The K--C metric is invariant under 
\begin{enumerate}
 \item the left action of $\fH$, $(z,t)\to(\zeta,s)*(z,t)$;
\item the Heisenberg translations $(z,t)\mapsto (z,t+s)$, $s\in\R$;
\item the rotations about the vertical axis $(z,t)\mapsto(ze^{i\phi},t)$, $\phi\in\R$.
\end{enumerate}
The group ${\rm Isom}(\fH,d_K)$ of {\it Heisenberg isometries}, is thus represented by the group consisting of matrices of the form
\begin{equation}
 \left[\begin{matrix}
        1&-\sqrt{2}\zeta e^{i\phi}&-|\zeta|^2+is\\
0&e^{i\phi}&\sqrt{2}\zeta\\
0&0&1
       \end{matrix}\right].
\end{equation}

The K--C metric is also scaled up to multiplicative constants by the action of Heisenberg dilations $(z,t)\mapsto$ $(rz,r^2t)$, $r\in\R_*$. The group
$$
{\rm Sim}(\fH,d_K)=\R\times {\rm U}(1)\times\fH
$$
acting on $\fH$ is called the group of Heisenberg similarities.

\subsubsection{$\R-$circles and $\C-$circles}

$\R-$circles are boundaries of Lagrangian planes and $\C-$circles are boundaries of complex lines. They come in two flavours, infinite ones (i.e. containing the point at infinity) and finite ones. We refer to \cite{Gol} for more more details about these curves.

\subsubsection{Cross--ratios}
Given a  quadruple of distinct points $\fp=(p_1,p_2,p_3,p_4)$ in $\partial\bH_\C^2$,  their complex  cross--ratio as defined by Kor\'anyi and Reimann in \cite{KR1} is 
$$
\X(p_1,p_2,p_3,p_4)=\frac{\langle \bp_3,\bp_1\rangle \langle \bp_4,\bp_2\rangle}{\langle \bp_4,\bp_1\rangle \langle \bp_3,\bp_2\rangle},
$$
where $\bp_i$ are lifts of $p_i$, $i=1,2,3,4$. The cross--ratio is independent of the choice of lifts and remains invariant under the diagonal action of ${\rm PU}(2,1)$. We stress here that for points in the Heisenberg group, the square root of its absolute value is
\begin{equation*}
 |\X(p_1,p_2,p_3,p_4)|^{1/2}=\frac{d_K(p_4,p_2)\cdot d_K(p_3,p_1)}{d_K(p_4,p_1)\cdot d_K(p_3,p_2)}.
\end{equation*}

\subsection{Cross--ratio variety}\label{sec:X-variety}

Given a quadruple $\fp=(p_1,p_2,p_3,p_4)$ of distinct points in the boundary, all possible permutations of points gives us 24 complex cross--ratios corresponding to $\fp$. Due to  symmetries, see \cite{F}, Falbel showed that all cross--ratios corresponding to a quadruple of points depend on three cross--ratios which satisfy two real equations. The following proposition holds.

\medskip

\begin{prop}\label{prop:cross-ratio-equalities}
Let $\fp=(p_1,p_2,p_3,p_4)$ be any quadruple of distinct points in $\partial \bH^2_\C$. Let
$$
\X_1(\fp)=\X(p_1,p_2,p_3,p_4),\quad \X_2(\fp)=\X(p_1,p_3,p_2,p_4),\quad \X_3(\fp)=\X(p_2,p_3,p_1,p_4).
$$
Then
\begin{eqnarray}\label{eq:cross1}
 &&
|\X_2|=|\X_1||\X_3|,\\
&&\label{eq:cross2}
2|\X_1|^2\Re(\X_3)=|\X_1|^2+|\X_2|^2-2\Re(\X_1+\X_2)+1.
\end{eqnarray}
\end{prop}

\medskip

For the proof, see for instance in \cite{PP}. The above equations define a 4--dimensional real subvariety of $\C^3$.\footnote{We note here that Equalities \ref{eq:cross1} and \ref{eq:cross2} also hold in the case of quadruples of distinct points in the boundary of the quaternionic hyperbolic space $\partial\bH_\h^2$. This is identified with the one point compactification of the quaternionic Heisenberg group, see  \cite{P-q}.}

\medskip

\begin{defn}\label{defn:cr-variety}
 The {\it cross--ratio variety} $\fX$ is the subset of $\C^3$ at which Eqs. \ref{eq:cross1} and \ref{eq:cross2} are satisfied.
\end{defn}

\medskip

\section{Ptolemaean Inequality and Ptolemaeus' Theorem}

In this section we prove an $S^3$ version of the Ptolemaean inequality and Ptolemaeus’ Theorem
respectively, which are derived almost immediately from the properties of complex cross--ratios.

\begin{thm}\label{thm:ptolemaean-ineq}{\bf (Ptolemaean inequality and Ptolemaeus' Theorem in $\partial\bH^2_\C)$}
Let $\fp=(p_1,p_2,$ $p_3,p_4)$ a quadruple of distinct points in $\partial \bH_\C^2$ and $\X_i=\X_i(\fp)$, $i=1,2,3,$  its corresponding complex cross--ratios.  Then  the following inequalities hold:
 \begin{equation}\label{eq:ptolemaean-inequality-cr}
  |\X_1|^{1/2}+|\X_2|^{1/2}\ge 1,\quad \text{and}\quad -1\le|\X_1|^{1/2}-|\X_2|^{1/2}\le 1.
 \end{equation}
Each inequality  \ref{eq:ptolemaean-inequality-cr} holds if and only if all four points of $\fp$ lie in an $\R-$circle. Then, $\X_i>0$, $i=1,2$ and 
\begin{enumerate}
 \item $\X_1^{1/2}-\X_2^{1/2}=1$ if $p_1$ and $p_3$ separate $p_2$ and $p_4$;
 \item $\X_2^{1/2}-\X_1^{1/2}=1$ if  $p_1$ and $p_2$ separate $p_3$ and $p_4$;
 \item $\X_1^{1/2}+\X_2^{1/2}=1$ if  $p_1$ and $p_4$ separate $p_2$ and $p_3$.
\end{enumerate}
\end{thm}
\begin{proof}
From the defining equations of $\fX$ we have
\begin{eqnarray*}
0&=& |\X_1|^2+|\X_2|^2-2\Re(\X_1)-2\Re(\X_2)+1-2|\X_1|^2\Re(\X_3)\\
&\ge & |\X_1|^2+|\X_2|^2-2|\X_1|-2|\X_2|+1-2|\X_1|^2|\X_3|\\
\text{from\;Eq.\ref{eq:cross1}}&=&|\X_1|^2+|\X_2|^2-2|\X_1|-2|\X_2|+1-2|\X_1||\X_2|\\
0&\ge&|\X_1|^2+|\X_2|^2-2\Re(\X_1)-2\Re(\X_2)+1-2|\X_1||\X_2|\\
&\ge&|\X_1|^2+|\X_2|^2-2|\X_1|-2|\X_2|+1-2|\X_1||\X_2|\\
&=&(|\X_1|+|\X_2|-1)^2-4|\X_1||\X_2|\\
&=&(|\X_1|+|\X_2|-2|\X_1|^{1/2}|\X_2|^{1/2}-1)\cdot (|\X_1|+|\X_2|+2|\X_1|^{1/2}|\X_2|^{1/2}-1)\\
&=&\left((|\X_1|^{1/2}-|\X_2|^{1/2})^2-1\right)\cdot\left((|\X_1|^{1/2}+|\X_2|^{1/2})^2-1\right),
\end{eqnarray*}
Therefore,
$$
(|\X_1|^{1/2}-|\X_2|^{1/2}-1)\cdot(|\X_1|^{1/2}-|\X_2|^{1/2}+1)\cdot(|\X_1|^{1/2}+|\X_2|^{1/2}-1)\cdot(|\X_1|^{1/2}+|\X_2|^{1/2}+1)\le 0
$$
and since $|\X_1|^{1/2}+|\X_2|^{1/2}+1>0$, this reduces to
$$
(|\X_1|^{1/2}-|\X_2|^{1/2}-1)\cdot(|\X_1|^{1/2}-|\X_2|^{1/2}+1)\cdot(|\X_1|^{1/2}+|\X_2|^{1/2}-1)\le 0.
$$
Suppose that $\left(|\X_1|^{1/2}-|\X_2|^{1/2}\right)^2> 1$; then applying triangle inequality we are reduced to a contradiction since we must also have $\left(|\X_1|^{1/2}-|\X_2|^{1/2}\right)^2> 1$.  Therefore, $\left(|\X_1|^{1/2}-|\X_2|^{1/2}\right)^2\le 1$ and $\left(|\X_1|^{1/2}-|\X_2|^{1/2}\right)^2\ge 1$ 
which proves the Ptolemaean inequality.

To prove Ptolemaeus' Theorem, suppose first that one of the inequalities holds as an equality. Then,
$$
(|\X_1|^{1/2}-|\X_2|^{1/2}-1)\cdot(|\X_1|^{1/2}-|\X_2|^{1/2}+1)\cdot(|\X_1|^{1/2}+|\X_2|^{1/2}-1)\cdot(|\X_1|^{1/2}+|\X_2|^{1/2}+1)= 0,
$$
which is equivalent to
$$
(|\X_1|-|\X_2|)^2= 2(|\X_1|+|\X_2|)-1.
$$
Since $
(|\X_1|-|\X_2|)^2\le 2\Re(\X_1+\X_2)-1,
$
we have $|\X_1|+|\X_2|\le \Re(\X_1+\X_2)$ and therefore
$$
0\ge\Re(\X_2)-|\X_2|\ge |\X_1|-\Re(\X_1)\ge 0.
$$
Thus $\X_1,\X_2$ are positive. Now from Equation \ref{eq:cross2} it follows
\begin{eqnarray*}
 2\X_1^2\Re(\X_3)&=&\X_1^2+\X_2^2-2\X_1-2\X_2+1\\
&=&\X_1^2+\X_2^2-(\X_1-\X_2)^2=2\X_1\X_2\\
\text{using \;Eq.}\; \ref{eq:cross1}\quad &=&2\X_1^2|\X_3|,
\end{eqnarray*}
and thus $\X_3>0$. By Proposition 5.12 (ii)  in \cite{PP} all points have to lie in a Lagrangian plane.

Conversely, if all points lie on an $\R$-circle, then we have equality in one of the inequalities, see Proposition 5.14 of \cite{PP}. The last statement is just Corollary 5.15 of \cite {PP}.
\end{proof}

\medskip

Let now $\fp=(p_1,p_2,p_3,p_4)$ a quadruple of distinct points in the Heisenberg group $\fH$. Inequalities \ref{eq:ptolemaean-inequality-cr} of Theorem \ref{thm:ptolemaean-ineq} can be written as
\begin{eqnarray}\label{eq:ineq-heis1}
&&
d_K(p_2,p_3)\cdot d_K(p_1,p_4)\le d_K(p_2,p_4)\cdot d_K(p_1,p_3)+d_K(p_1,p_2)\cdot d_K(p_3,p_4),\\
&&\label{eq:ineq-heis2}
d_K(p_1,p_3)\cdot d_K(p_2,p_4)\le d_K(p_1,p_2)\cdot d_K(p_3,p_4)+d_K(p_2,p_3)\cdot d_K(p_1,p_4),\\
&&\label{eq:ineq-heis3}
d_K(p_1,p_2)\cdot d_K(p_3,p_4)\le d_K(p_1,p_3)\cdot d_K(p_2,p_4)+d_K(p_2,p_3)\cdot d_K(p_1,p_4).
\end{eqnarray}

\medskip

As a corollary of the above discussion we derive the following.

\medskip

\begin{thm}\label{thm:ptolemaean-heis} {\bf (Ptolemaean inequality and  Ptolemaeus' Theorem in the Heisenberg group)}
 The Kor\'anyi--Cygan metric $d_K$ in the Heisenberg group $\fH$ satisfies the Ptolemaean inequality: for each quadruple of points $\fp=(p_1,p_2,p_3,p_4)$, inequalities \ref{eq:ineq-heis1}, \ref{eq:ineq-heis2} and \ref{eq:ineq-heis3} hold. Moreover, each of these inequalities hold as an equality if and only if all points lie in an $\R-$circle. Explicitly,
\begin{enumerate}
\item $d_K(p_2,p_3)\cdot d_K(p_1,p_4)=d_K(p_2,p_4)\cdot d_K(p_1,p_3)+d_K(p_1,p_2)\cdot d_K(p_3,p_4)$ if and only if all points lie in an $\R-$circle and $p_1$ and $p_4$ separate $p_2$ and $p_3$;
\item $d_K(p_1,p_3)\cdot d_K(p_2,p_4)=d_K(p_1,p_2)\cdot d_K(p_3,p_4)+d_K(p_2,p_3)\cdot d_K(p_1,p_4)$ if and only if all points lie in an $\R-$circle and $p_1$ and $p_3$ separate $p_2$ and $p_4$;
\item $d_K(p_1,p_2)\cdot d_K(p_3,p_4)=d_K(p_1,p_3)\cdot d_K(p_2,p_4)+d_K(p_2,p_3)\cdot d_K(p_1,p_4)$ if and only if all points lie in an $\R-$circle and $p_1$ and $p_2$ separate $p_3$ and $p_4$.
\end{enumerate}
\end{thm}

\end{document}